\documentclass[11pt,reqno]{amsart}

\newtheorem{proposition}{Proposition}[section]
\newtheorem{lemma}[proposition]{Lemma}
\newtheorem{corollary}[proposition]{Corollary}
\newtheorem{theorem}[proposition]{Theorem}

\theoremstyle{definition}

\newtheorem{remark}[proposition]{Remark}

\newcommand{\thlabel}[1]{\label{th:#1}}
\newcommand{\thref}[1]{Theorem~\ref{th:#1}}
\newcommand{\selabel}[1]{\label{se:#1}}
\newcommand{\seref}[1]{Section~\ref{se:#1}}
\newcommand{\lelabel}[1]{\label{le:#1}}
\newcommand{\leref}[1]{Lemma~\ref{le:#1}}
\newcommand{\prlabel}[1]{\label{pr:#1}}
\newcommand{\prref}[1]{Proposition~\ref{pr:#1}}

\newcommand{\eqlabel}[1]{\label{eq:#1}}
\newcommand{\equref}[1]{(\ref{eq:#1})}

\newcommand{\Aut}{{\rm Aut}\,}

\def\ot{\otimes}

\newcommand{\Cc}{\mathcal{C}}

\def\*C{{}^*\hspace*{-1pt}{\Cc}}
\def\text#1{{\rm {\rm #1}}}

\input xy
\xyoption {all} \CompileMatrices

\usepackage{color,amssymb,graphicx,amscd}

\begin{document}

\title[Classifying bicrossed products of two Sweedler's Hopf algebras]
{Classifying bicrossed products of two Sweedler's Hopf algebras}

\author{Costel-Gabriel Bontea}
\address{Faculty of Mathematics and Computer Science, University of Bucharest, Str.
Academiei 14, RO-010014 Bucharest 1, Romania}
\address{Faculty of Engineering, Vrije Universiteit Brussel, Pleinlaan 2, B-1050 Brussels, Belgium}
\email{costel.bontea@gmail.com}

\thanks{ This work was supported by a grant of the Romanian National
Authority for Scientific Research, CNCS-UEFISCDI, grant no.
88/05.10.2011.}



\begin{abstract}
In this paper we continue the study started recently in
\cite{ABMbp} by describing and classifying all Hopf algebras $E$
that factorize through two Sweedler's Hopf algebras. Equivalently,
we classify all bicrossed products $H_4 \bowtie H_4$. There are
three steps in our approach. First, we explicitly describe the set
of all matched pairs $(H_4, H_4, \triangleright, \triangleleft)$
by proving that, with the exception of the trivial pair, this set
is parameterized by the ground field $k$. Then, for any $\lambda
\in k$, we describe by generators and relations the associated
bicrossed product, $\mathcal{H}_{16, \, \lambda}$. This is a
$16$-dimensional, pointed, unimodular and non-semisimple Hopf
algebra. A Hopf algebra $E$ factorizes through $H_4$ and $H_4$ if
and only if $ E \cong H_4 \ot H_4$ or $E \cong {\mathcal H}_{16,
\, \lambda}$. In the last step we classify these quantum groups by
proving that there are only three isomorphism classes represented
by: $H_4 \ot H_4$, ${\mathcal H}_{16, \, 0}$ and ${\mathcal
H}_{16, \, 1} \cong D(H_4)$, the Drinfel'd double of $H_4$. The
automorphism group of these objects is also computed: in
particular, we prove that $\Aut_{\rm Hopf}\big( D(H_4)\big)$ is
isomorphic to a semidirect product of groups, $k^{\times} \rtimes
\mathbb{Z}_2$.
\end{abstract}
\maketitle

\section*{Introduction}

Let $A$ and $H$ be two given Hopf algebras. The factorization
problem for Hopf algebras consists of classifying up to an
isomorphism all Hopf algebras that factorize through $A$ and $H$,
i.e. all Hopf algebras $E$ containing $A$ and $H$ as Hopf
subalgebras such that the multiplication map $A \ot H \to E$, $a
\ot h \mapsto a h$ is bijective. The problem can be put in more
general terms but we restrict ourselves to the case of Hopf
algebras. For a detailed account on the subject the reader may
consult \cite{ABMbp}.

An important step in dealing with the factorization problem was
made by Majid in \cite[Proposition 3.12]{majid} who generalized to
Hopf algebras the construction of the bicrossed product for groups
introduced by Takeuchi in \cite{Takeuchi}. Although in
\cite{majid} the construction is known under the name of double
cross product, we will follow \cite{Kassel} and call it, just like
in the case of groups, the bicrossed product construction. A
bicrossed product of two Hopf algebras $A$ and $H$ is a new Hopf
algebra $A \bowtie H$ associated to a matched pair $(A, H, \rhd,
\lhd )$ of Hopf algebras. It is proven in \cite[Proposition
3.12]{majid} that a Hopf algebra $E$ factorizes through $A$ and
$H$ if and only if $E$ is isomorphic to some bicrossed product of
$A$ and $H$. Thus, the factorization problem can be stated in a
computational manner: for two Hopf algebras $A$ and $H$ describe
the set of all matched pairs $(A, H, \rhd, \lhd )$ and classify up
to an isomorphism all the bicrossed products $A \bowtie H$. This
way of approaching the problem was recently proposed in
\cite{ABMbp} with promising results regarding new examples of
quantum groups. For example, in \cite[Section 4]{{ABMbp}} all
bicrossed products $H_4 \bowtie k[C_n]$ are described by
generators and relations and are classified. They are quantum
groups at roots of unity $H_{4n, \, \omega}$ which are classified
by the arithmetic of the ring $\mathbb{Z}_n$. In this paper we
continue the study began in \cite{ABMbp} by classifying all Hopf
algebras that factorize through two Sweedler's Hopf algebras.

The paper is organized as follows. In \seref{prel} we set the
notations and recall the bicrossed product construction of two
Hopf algebras. In \seref{H_4H_4}, the main section of this paper,
we classify all Hopf algebras that factorize through two
Sweedler's Hopf algebras. For this, we first compute all the
matched pairs $(H_4, H_4, \rhd, \lhd)$: except the trivial one,
these are parameterized by the ground field $k$, which is an
arbitrary field of characteristic $\neq 2$. We then describe by
generators and relations the associated bicrossed products $H_4
\bowtie H_4$. These are: $H_4 \ot H_4$ and ${\mathcal H}_{16, \,
\lambda}$, where, for any $\lambda \in k$, ${\mathcal H}_{16, \,
\lambda}$ is the $16$-dimensional quantum group generated by $g$,
$x$, $G$, $X$ subject to the relations:
$$
g^2 = G^2 = 1 \quad  x^2 = X^2 = 0, \quad gx = -xg, \quad GX =
-XG,
$$
$$
gG = Gg, \quad gX = -Xg, \quad x G = - Gx, \quad xX + Xx = \lambda
\, (1 -  Gg)
$$
with the coalgebra structure given such that $g$ and $G$ are
group-likes, $x$ is $(1, g)$-primitive and $X$ is $(1,
G)$-primitive. We then prove that there are only three isomorphism
classes of Hopf algebras that factorize through two Sweedler's
Hopf algebras: $H_4 \ot H_4$, ${\mathcal H}_{16, \, 0}$ and
${\mathcal H}_{16, \, 1} \cong D(H_4)$, the Drinfel'd double of
$H_4$. Finally, we prove that there exist the following
isomorphisms of groups:
$$
\Aut_{\rm Hopf}( D(H_4)) \cong k^{\times} \rtimes \mathbb{Z}_2,
\quad \Aut _{\rm Hopf}(\mathcal{H}_{16,0}) \cong (k^{\times}
\times k^{\times}) \rtimes \mathbb{Z}_2 \cong \Aut _{\rm Hopf}(H_4
\otimes H_4)
$$

Since the tensor product of two pointed coalgebras is pointed
\cite[Lemma 5.1.10]{Montgomery}, it follows that the bicrossed
product of two pointed Hopf algebras is pointed. In particular,
$H_4 \ot H_4$ and ${\mathcal H}_{16, \, \lambda}$, for $\lambda
\in k$, are pointed Hopf algebras of dimension 16. The
classification of such Hopf algebras, over an algebraically closed
field of characteristic zero, was considered in \cite{Caenepeel
Dascalescu} and, without the pointedness assumption, in \cite{GV}.
With the notations of \cite[Theorem 5.2]{Caenepeel Dascalescu}, we
have $\mathbb{H}_{4} \ot H_{4} = H_{(3)}$, $\mathcal{H}_{16, \, 0}
= H_{(4)}$ and $\mathcal{H}_{16, \, 1} = H_{(5)}$. Thus, the Hopf
algebras we obtain here have already appeared in the literature.

The classification of pointed Hopf algebras has been the subject
of an intense study in the past years (\cite{AS98},
\cite{Caenepeel Dascalescu}, \cite{AS00}, \cite{AS10}) and many
classification results are known, especially when the coradical is
commutative. The most impressive result of this type was obtained
by Andruskiewitsch and Schneider in \cite{AS10} where the
classification of all finite-dimensional pointed Hopf algebras
with commutative coradical, whose dimension is not divisible by
primes $\leq 7$, is given. Therefore, if one hopes to obtain
really new examples of Hopf algebras by considering the
factorization problem then one has better chances in succeeding if
he considers pointed Hopf algebras with non-commutative coradical
or non-pointed Hopf algebras.

Finally, we point out that the dual problem of classifying all the
extensions of $H_{4}$ by $H_{4}$ was solved by Garc\'ia and Vay in
\cite[Lemma 2.8]{GV} who showed that all such extensions are
isomorphic to the tensor product $H_{4} \ot H_{4}$. Their proof
uses the cocycle bicrossproduct construction of \cite{Maj90} and
\cite{AD} as a tool, but, instead of computing all the cocycle
bicrossproducts of $H_{4}$ and $H_{4}$, they build their argument
on the fact that an extension of a Hopf algebra $A$ by another
Hopf algebra $B$ is a $B$-cleft extension of $A$, which allows
them to use of the description and the classification of the
$H_{4}$-cleft extensions of an algebra $A$ from \cite{Mas94} and
\cite{DT95}. The link between the factorization problem and the
extension problem was observed by Majid, who shows in
\cite[Proposition 7.2.4]{majid2} that the set of matched pairs
$(A, H, \rhd, \lhd)$ is in bijection with the set of
bicrossproduct data $(A^{\ast}, H, \alpha, \beta)$ giving rise to
cocycle bicrossproducts with trivial cocycles. Since $H_{4}^{\ast}
\simeq H_{4}$ one sees that the above correspondence breaks down
at the level of the isomorphism classes of the associated Hopf
algebra products, which is not so surprising considering that the
two kinds of products are different objects.


\section{Preliminaries}\selabel{prel}

We work over an arbitrary field $k$ of characteristic $\neq 2$.
All algebras, coalgebras, Hopf algebras are over $k$ and $\ot =
\ot_k$. We shall use the standard notations from Hopf algebras
theory: in particular, for a coalgebra $C$, we use the
$\Sigma$-notation: $\Delta(c) = c_{(1)} \ot c_{(2)}$, for any
$c\in C$ (summation understood). Let $A$ and $H$ be two Hopf
algebras. $A$ is a called a left $H$-module coalgebra if there
exists $\rhd : H \ot A \to A$ a morphism of coalgebras such that
$(A, \rhd)$ is also a left $H$-module. Similarly, $H$ is called a
right $A$-module coalgebra if there exists $\lhd : H \otimes A
\rightarrow H$ a morphism of coalgebras such that $(H, \lhd) $ is
a right $A$-module. The actions $\rhd: H \otimes A \to A$ and
$\lhd : H \otimes A \to H$ are called \emph{trivial} if $h \rhd a
= \varepsilon_H(h) a$ and $h \lhd a = \varepsilon_A (a) h$
respectively, for all $a \in A$ and $h \in H$.

A \textit{matched pair} of Hopf algebras (\cite{majid},
\cite{Kassel}) is a quadruple $(A, H, \rhd, \lhd)$, where $A$ and
$H$ are Hopf algebras, $\rhd: H \otimes A \to A$ and $\lhd: H
\otimes A \to H$ are coalgebra maps such that $(A, \rhd)$ is a
left $H$-module coalgebra, $(H, \lhd)$ is a right $A$-module
coalgebra and the following compatibility conditions hold:
\begin{eqnarray}
h \rhd 1_{A} &{=}& \varepsilon_{H}(h)1_{A}, \,\,\, 1_{H} \lhd a
= \varepsilon_{A}(a)1_{H} \eqlabel{mp1}\\
h \rhd (ab) &{=}& (h_{(1)} \rhd a_{(1)}) \bigl ( (h_{(2)} \lhd
a_{(2)}) \rhd b \bigl)
\eqlabel{mp2} \\
(g h) \lhd a &{=}& \bigl( g \lhd (h_{(1)} \rhd a_{(1)}) \bigl)
(h_{(2)} \lhd a_{(2)})
\eqlabel{mp3} \\
h_{(1)} \lhd a_{(1)} \otimes h_{(2)} \rhd a_{(2)} &{=}& h_{(2)}
\lhd a_{(2)} \otimes h_{(1)} \rhd a_{(1)} \eqlabel{mp4}
\end{eqnarray}
for all $a$, $b\in A$, $g$, $h\in H$. If $(A, H, \rhd, \lhd)$ is a
matched pair of Hopf algebras then the associated
\textit{bicrossed product} $A \bowtie H$ of $A$ with $H$ is the
vector space $A \ot H$ endowed with the tensor product coalgebra
structure and the multiplication
\begin{equation}\eqlabel{0010}
(a \bowtie g) \cdot (b \bowtie h):= a (g_{(1)} \rhd b_{(1)})
\bowtie (g_{(2)} \lhd b_{(2)}) h
\end{equation}
for all $a$, $b\in A$, $g$, $h \in H$, where we use $\bowtie$ for
$\ot$. $A \bowtie H$ is a Hopf algebra with the antipode given by
the formula:
\begin{equation}\eqlabel{antipbic}
S ( a \bowtie h ) := \big( 1_A \bowtie S_H (h) \big) \cdot \big(
S_A (a) \bowtie 1_H \big)
\end{equation}
for all $a\in A$ and $h\in H$ \cite[Theorem 7.2.2]{majid2},
\cite[Theorem IX 2.3]{Kassel}.

The basic example of a bicrossed product is the famous Drinfel'd
double of a finite dimensional Hopf algebra $H$: $D(H) =
(H^*)^{\rm cop} \bowtie H$, the bicrossed product associated to a
given canonical matched pair \cite[Theorem IX.3.5]{Kassel}. For
others examples of bicrossed products we refer to \cite{ABMbp},
\cite{Kassel}, \cite{majid2}.

We recall that a Hopf algebra $E$ \emph{factorizes} through two
Hopf algebras $A$ and $H$ if there exist injective Hopf algebra
maps $i : A \to E$ and $j : H \to E$ such that the map
$$
A \ot H \to E, \quad a \ot h \mapsto i(a) j(h)
$$
is bijective. The next fundamental result is due to Majid
\cite[Proposition 3.12]{majid}:  A Hopf algebra $E$ factorizes
through two given Hopf algebras $A$ and $H$ if and only if there
exists a matched pair of Hopf algebras $(A, H, \rhd, \lhd)$ such
that $E \cong A \bowtie H$. In light of this result, the
factorization problem for Hopf algebras was restated \cite{ABMbp}
in a computational manner: for two given Hopf algebras, $A$ and
$H$, describe the set of all matched pairs $(A, H, \rhd, \lhd)$
and classify up to isomorphisms all bicrossed products $A \bowtie
H$.


\section{The bicrossed products of two Sweedler's Hopf algebras}\selabel{H_4H_4}

In this section we are going to classify all the bicrossed
products $H_4 \bowtie H_4$. Recall that Sweedler's $4$-dimensional
Hopf algebra, $H_4$, is generated by two elements, $g$ and $x$,
subject to the relations $g^{2} = 1$, $x^{2} = 0$ and $xg = -gx$.
The coalgebra structure and the antipode are given by:
$$
\Delta(g) = g \otimes g, \quad \varepsilon(g) = 1, \quad S(g) = g
$$
$$
\Delta(x) = x \otimes 1 + g \otimes x, \quad \varepsilon(x) = 0,
\quad S(x) = -gx
$$
In order to avoid confusions we will denote by $\mathbb{H}_4$ a
copy of $H_4$, and by $G$ and $X$ the generators of
$\mathbb{H}_4$. Thus, $G^{2} = 1$, $X^{2} = 0$, $GX = - XG$, $G$
is a group-like element and $X$ is an $(1, G)$-primitive element.

Recall that, for a Hopf algebra, $H$, $\mbox{G}(H) = \{ g \in H
\mid \Delta (g) = g \ot g, \, \varepsilon (g) = 1\}$ is the set of
group-like elements of $H$ and, for $g$,  $h \in \mbox{G}(H)$,
$\mbox{P}_{g, h} (H) = \{ x \in H \mid \Delta(x) = x \ot g + h \ot
x\}$ is the set of $(g, h)$-primitive elements of $H$. For the
Sweedler Hopf algebra we have:
$$
\textnormal{G}(H_4) = \{1, \, g \}, \quad \textnormal{P}_{1, \, 1}
(H_4) = \textnormal{P}_{g, \, g} (H_4) = \{0\}, \quad
\textnormal{P}_{1, g} (H_4) = k (1 - g) \oplus kx
$$

The next theorem describes the set of all matched pairs $(A =
\mathbb{H}_4, H = H_4, \rhd, \lhd)$.

\begin{theorem}\thlabel{mpHH00}
Let $k$ be a field of characteristic $\neq 2$. Then
$(\mathbb{H}_4, H_4, \rhd , \lhd)$ is a matched pair of Hopf
algebras if and only if $(\rhd$, $\lhd)$ are both the trivial
actions or the pair $(\rhd$, $\lhd)$ is given by:
\begin{center}
\begin{tabular} {r | r  r  r  r  }
$\rhd$ & 1 & $G$ & $X$ & $GX$\\
\hline
1 & 1 & $G$ & $X$ & $GX$ \\
$g$ & 1 & $G$ & $-X$ & $-GX$ \\
$x$ & 0 & 0 & $\lambda - \lambda \, G$ & $\lambda - \lambda \, G$\\
$gx$ & 0 & 0 & $\lambda - \lambda \, G$ & $\lambda - \lambda\,G$
\\\end{tabular} \, \qquad
\begin{tabular} {r | r  r  r  r  }
$\lhd$ & 1 & $G$ & $X$ & $GX$\\
\hline
1 & 1 & 1 & 0 & 0\\
$g$ & $g$ & $g$ & 0 & 0 \\
$x$ & $x$ & $-x$ & $\lambda - \lambda \, g$ & $-\lambda + \lambda \, g$\\
$gx$ & $gx$ & $-gx$ & $-\lambda + \lambda \, g$ & $\lambda - \lambda \, g$ \\
\end{tabular}
\end{center}
for some $\lambda \in k$.
\end{theorem}

We prove this result in three steps. The first one is
\leref{actdr} where we describe the set of all right
$\mathbb{H}_4$-module coalgebra structures $\lhd$ on $H_4$
satisfying the normalizing condition $1 \lhd h = \varepsilon(h)1$,
for all $h \in \mathbb{H}_4$. There will be four such families of
actions, $\lhd^{j}$, $j = 1,2,3,4$, parameterized by scalars $a$,
$b$, $c$, $d\in k$.

The second step is \leref{actstg} where we describe the set of all
left $H_4$-module coalgebra structures $\rhd$ on $\mathbb{H}_4$
satisfying the normalizing condition $h \rhd 1 = \varepsilon(h)1$,
for all $h \in H_4$. There will also be four families of such
actions, $\rhd^{i}$, $i = 1,2,3,4$ parameterized by scalars $s$,
$t$, $u$, $v\in k$.

The final step consists of a detailed analysis of the sixteen
possibilities of choice for the pair of actions $(\rhd^i,
\lhd^j)$, for all $i$, $j = 1$, $2$, $3$, $4$. This will show that
the only ones that verify the axioms \equref{mp2}-\equref{mp4} of
a matched pair are: $(\rhd^1, \lhd^1)$, i.e. the pair of trivial
actions, and $(\rhd^4, \lhd^4)$, in which case the actions take
the form described in the statement.

We begin with:

\begin{lemma}\lelabel{actdr}
If $\lhd : H_4 \ot \mathbb{H}_4 \rightarrow H_4$ is a right
$\mathbb{H}_4$-module coalgebra structure such that $1 \lhd h =
\varepsilon(h)1$, for all $h \in \mathbb{H}_4$, then $\lhd$ has
one of the following forms:
\begin{center}
\begin{tabular} {l | r  r  r  r  }
$\lhd^1$ & 1 & $G$ & $X$ & $GX$\\
\hline 1 & 1 & 1 & 0 & 0\\
$g$ & $g$ & $g$ & 0 & 0 \\
$x$ & $x$ & $x$ & 0 & 0\\
$gx$ & $gx$ & $gx$ & 0 & 0 \\
\end{tabular} \qquad
\begin{tabular} {l | r  r  r  r  }
$\lhd^2$ & 1 & $G$ & $X$ & $GX$\\
\hline 1 & 1 & 1 & 0 & 0\\
$g$ & $g$ & $g$ & 0 & 0 \\
$x$ & $x$ & $x$ & 0 & 0\\
$gx$ & $gx$ & $c-cg-gx$ & $d-dg$ & $-d+dg$ \\
\end{tabular}
\end{center}
\vspace{0.5mm}
\begin{center}
\begin{tabular} {l | r  r  r  r  }
$\lhd^3$ & 1 & $G$ & $X$ & $GX$\\
\hline 1 & 1 & 1 & 0 & 0\\
$g$ & $g$ & $g$ & 0 & 0 \\
$x$ & $x$ & $a-ag-x$ & $b-bg$ & $-b+bg$\\
$gx$ & $gx$ & $gx$ & 0 & 0 \\
\end{tabular} \qquad
\begin{tabular} {l | r  r  r  r  }
$\lhd^4$ & 1 & $G$ & $X$ & $GX$\\
\hline 1 & 1 & 1 & 0 & 0\\
$g$ & $g$ & $g$ & 0 & 0 \\
$x$ & $x$ & $a-ag-x$ & $b-bg$ & $-b+bg$\\
$gx$ & $gx$ & $c-cg-gx$ & $d-dg$ & $-d+dg$ \\
\end{tabular}
\end{center}
where $a$, $b$, $c$, $d\in k$.
\end{lemma}

\begin{proof} Let $\lhd : H_4 \ot \mathbb{H}_4 \rightarrow H_4$ be
a right $\mathbb{H}_4$-module coalgebra structure such that $1
\lhd h = \varepsilon(h)1$, for all $h \in \mathbb{H}_4$. Then $1
\lhd G = 1$, $1 \lhd X = 0$, and $1 \lhd (GX) = 0$. Also, $g \lhd
G \in \mbox{G} (H_{4})$. We cannot have $g \lhd G = 1$, for
otherwise $1 = 1 \lhd G = (g \lhd G) \lhd G = g \lhd 1 = g$.
Therefore $g \lhd G = g$. Since $g \lhd X \in \mbox{P}_{g, g}
(H_{4})$, we deduce that $g \lhd X = 0$. Similarly, $g \lhd (GX) =
0$. Observe that the actions of $X$ and $GX$ on $g$ are compatible
with the relations $X^2 = 0$ and $G X =-XG$.

We next show that
\begin{center}
\begin{tabular}{l | r  r  r  r }
  $\lhd$ & 1 & $G$ & $X$ & $GX$ \\
  \hline
  $x$ & $x$ & $x$ & 0 & 0 \\
\end{tabular} \hspace{0.2in} or \hspace{0.2in}
\begin{tabular}{l | r  r  r  r }
  $\lhd$ & 1 & $G$ & $X$ & $GX$ \\
  \hline
  $x$ & $x$ & $a - ag - x$ & $b - bg$ & $-b + bg$ \\
\end{tabular}
\end{center}
for some $a$, $b\in k$.

We have $x \lhd G \in \mbox{P}_{1,g}(H_4)$, hence $x \lhd G = a -
ag + bx$, for some $a$, $b \in k$. Since the action of $G$ is
compatible with $G^2 = 1$, we have
$$
x = x \lhd 1 = (x \lhd G) \lhd G = (a - ag + bx) \lhd G = a + ba -
(a + ba)g + b^2x
$$
Thus, $b^2 = 1$ and $a (1 + b) = 0$. If $b = -1$ then there are no
restrictions on $a$. Otherwise, $b = 1$ and $a = 0$. This shows
that $x \lhd G = x$ or $x \lhd G = a - ag - x$, with $a \in k$.

We also have $x \lhd X \in \mbox{P}_{1,g} (H_4)$, hence $x \lhd X
= b - bg + cx$, for some $b$, $c \in k$. Using that $X^2=0$, we
have
$$
0 = x \lhd 0 = (x \lhd X) \lhd X = (b - bg + cx) \lhd X = cb - cbg
+ c^2x
$$
Thus, $c = 0$ and $x \lhd X = b - bg$.

If $x \lhd G = x$ then $ x \lhd (GX) = x \lhd X$ and, if $x \lhd G
= a - ag - x$ then $ x \lhd (GX) = (a - ag - x) \lhd X = -x \lhd
X$. Observe that, in both cases, $\varepsilon \big( x \lhd (GX)
\big) = 0$, and $ \Delta \big( x \lhd (GX) \big) = x_{(1)} \lhd
(GX)_{(1)} \ot x_{(2)} \lhd (GX)_{(2)}$.

It remains to see when $x \lhd (GX) = x \lhd (-XG)$. If $x \lhd G
= x$, then
$$
x \lhd (XG) = (b - bg) \lhd G = b - bg = x \lhd X = x \lhd (GX)
$$
Thus, $x \lhd (GX) = x \lhd (-XG)$ implies $x \lhd X = x \lhd (GX)
= 0$. If $x \lhd G = a - ag - x$, then
$$
x \lhd (XG) = (b - bg) \lhd G = b - bg = x \lhd X = -x \lhd (GX)
$$
In this case, the equality $x \lhd (GX) = x \lhd (-XG)$ is
satisfied without further restrictions.

In a similar manner it can be shown that
\begin{center}
\begin{tabular}{l | r  r  r  r }
  $\lhd$ & 1 & $G$ & $X$ & $GX$ \\
  \hline
  $gx$ & $gx$ & $gx$ & 0 & 0 \\
\end{tabular} \hspace{0.2in} or \hspace{0.2in}
\begin{tabular}{l | r  r  r  r }
  $\lhd$ & 1 & $G$ & $X$ & $GX$ \\
  \hline
  $gx$ & $gx$ & $c - cg - gx$ & $d - dg$ & $-d + dg$ \\
\end{tabular}
\end{center}
for some $c$, $d\in k$.
\end{proof}

Analogous to \leref{actdr} we can prove:

\begin{lemma}\lelabel{actstg}
If $\rhd : H_4 \ot \mathbb{H}_4 \to \mathbb{H}_4$ is a left
$H_4$-module coalgebra structure such that $h \rhd 1 =
\varepsilon(h)1$, for all $h \in H_4$, then $\rhd$ has one of the
following forms:
\begin{center}
\begin{tabular} {l | r  r  r  r  }
$\rhd^1$ & 1 & $G$ & $X$ & $GX$\\
\hline 1 & 1 & $G$ & $X$ & $GX$\\
$g$ & 1 & $G$ & $X$ & $GX$\\
$x$ & 0 & 0 & 0 & 0\\
$gx$ & 0 & 0 & 0 & 0 \\
\end{tabular} \qquad
\begin{tabular} {l | r  r  r  r  }
$\rhd^2$ & 1 & $G$ & $X$ & $GX$\\
\hline 1 & 1 & $G$ & $X$ & $GX$\\
$g$ & 1 & $G$ & $X$ & $u-uG-GX$ \\
$x$ & 0 & 0 & 0 & $v-vG$\\
$gx$ & 0 & 0 & 0 & $v-vG$ \\
\end{tabular}
\end{center}
\vspace{0.5mm}
\begin{center}
\begin{tabular} {l | r  r  r  r  }
$\rhd^3$ & 1 & $G$ & $X$ & $GX$\\
\hline 1 & 1 & $G$ & $X$ & $GX$\\
$g$ & 1 & $G$ & $s-sG-X$ & $GX$ \\
$x$ & 0 & 0 & $t-tG$ & 0\\
$gx$ & 0 & 0 & $t-tG$ & 0 \\
\end{tabular} \qquad
\begin{tabular} {l | r  r  r  r  }
$\rhd^4$ & 1 & $G$ & $X$ & $GX$\\
\hline 1 & 1 & $G$ & $X$ & $GX$\\
$g$ & 1 & $G$ & $s-sG-X$ & $u-uG-GX$ \\
$x$ & 0 & 0 & $t-tG$ & $v-vG$\\
$gx$ & 0 & 0 & $t-tG$ & $v-vG$ \\
\end{tabular}
\end{center}
where $s$, $t$, $u$, $v\in k$.
\end{lemma}

\begin{proof} One can check the validity of the statement by
employing the same arguments as those used in the proof of
\leref{actdr}. A more elegant and shorter proof can be deduced
from the following elementary remark: if $\rhd: H \ot C \to C $ is
a left $H$-module coalgebra on $C$ then $\lhd : C \ot H^{\rm cop}
\to C$, $c \lhd h := S(h) \rhd c$, for all $c \in C$ and $h \in H$
is a right $H^{\rm cop}$-module coalgebra on $C$, and the above
correspondence is bijective if the antipode of $H$ is bijective.
We apply this observation for $H = H_4$ and $C = \mathbb{H}_4$,
which is just a copy of $H_4$, taking into account that the
antipode of $H_4$ is bijective and $H_4^{\rm cop} \cong H_4$. In
this way, the proof of \leref{actstg} follows from the one of
\leref{actdr}.
\end{proof}

We are now in a position to finish the proof of \thref{mpHH00}.

\begin{proof}[The proof of \thref{mpHH00}]
Let $(\mathbb{H}_4, H_4, \rhd, \lhd)$ be a matched pair. Since
$\lhd : H_4 \ot \mathbb{H}_4 \to H_4$ is a right
$\mathbb{H}_4$-module coalgebra structure satisfying $1 \lhd h =
\varepsilon(h)1$, for all $h \in \mathbb{H}_4$, we deduce from
\leref{actdr} that $\lhd$ is one of the $\lhd^i$'s. Similarly,
$\rhd : H_4 \ot \mathbb{H}_4 \to \mathbb{H}_4$ is a left
$H_4$-module coalgebra structure satisfying $h \rhd 1 =
\varepsilon(h)1$, for all $h\in H_4$, hence, $\rhd$ is one of the
$\rhd^j$'s, by \leref{actstg}. We next show that $(\mathbb{H}_4,
H_4,\rhd^j, \lhd^i)$ is a matched pair if and only if
$(i,j)\in\{(1,1),(4,4)\}$ and, if $(i,j)=(4,4)$, then $\rhd^i$ and
$\lhd^j$ are defined as we have claimed.

Firstly, if $i = 2, 3$ or $j = 2, 3$ then $(\mathbb{H}_4,
H_4,\rhd^i, \lhd^j)$ is not a matched pair. Indeed, if $i = 2, 3$
then condition \equref{mp2} is not satisfied for the triple $(g,
G, X)$, while if $j = 2, 3$ then condition \equref{mp3} is not
satisfied for the triple $(x, g, G)$.

Secondly, $(\mathbb{H}_4, H_4, \rhd^4, \lhd^1)$ and
$(\mathbb{H}_4, H_4, \rhd^1, \lhd^4)$ are not matched pairs, since
condition \equref{mp4} fails to be fulfilled in the former case
for the pair $(x, GX)$ and in the later case for the pair $(gx,
X)$.

We focus now our attention on when $(\mathbb{H}_4, H_4, \rhd^4,
\lhd^4)$ is a matched pair and for this we look at the conditions
\equref{mp2}-\equref{mp4}. It is not hard to see that \equref{mp4}
is trivially fulfilled for all $(h,a) \in \{1, g, x, gx \} \times
\{1, G, X, GX \} \setminus \{(x, X), (x, GX), (gx, X), (gx,
GX)\}$. A straightforward computation shows that the same
condition is satisfied by $(x, X)$ if and only if $t = b$ and $a =
s = 0$, by $(gx, GX)$ if and only if $v = -d$ and $c = u = 0$, by
$(x, GX)$ if and only if $v = b$, and by $(gx, X)$ if and only if
$t = -d$. Thus, condition \equref{mp4} is fulfilled if and only if
$a = c = s = u = 0$, $t = v = b$ and $d = -b$. It remains to see
that conditions \equref{mp2} and \equref{mp3} are compatible with
the relations $G^2 = g^2 = 1$, $X^2 = x^2 = 0$, $gx = -xg$ and $GX
= -XG$. Since this is straightforward, the proof is complete.
\end{proof}

We are able to describe and classify all Hopf algebras that
factorize through two Sweedler's Hopf algebras.

\begin{theorem}\thlabel{clH_4H_4}
Let $k$ be a field of characteristic $\neq 2$. Then:

$(1)$ A Hopf algebra $E$ factorizes through $\mathbb{H}_4$ and
$H_4$ if and only if $E \cong \mathbb{H}_4 \ot H_4$ or $E \cong
{\mathcal H}_{16, \, \lambda}$, for some $\lambda \in k$, where
${\mathcal H}_{16, \, \lambda}$ is the $16$-dimensional Hopf
algebra generated by $g$, $x$, $G$, $X$ subject to the relations:
$$
g^2 = G^2 = 1 \quad  x^2 = X^2 = 0, \quad gx = -xg, \quad GX = -XG
$$
$$
gG = Gg, \quad gX = -Xg, \quad x G = - Gx, \quad xX + Xx = \lambda
\, (1 -  Gg)
$$
with the coalgebra structure given by
$$
\Delta (g) = g \ot g, \quad \Delta (x) = x\ot 1 + g \ot x, \quad
\Delta (G) = G \ot G, \quad \Delta (X) = X\ot 1 + G \ot X,
$$
$$
\varepsilon (g) = \varepsilon (G) = 1, \qquad \varepsilon (x) =
\varepsilon (X) = 0
$$

$(2)$ $\mathcal{H}_{16, \, \lambda}$ is pointed, unimodular, and
non-semisimple. Moreover,
\begin{center}
$\textnormal{P}_{1, g} \big( \mathcal{H}_{16, \, \lambda} \big) =
k (1 - g) \oplus kx, \qquad \textnormal{P}_{1, G} \big(
\mathcal{H}_{16, \, \lambda} \big) = k (1 - G) \oplus kX,$

$\textnormal{P}_{1, gG} \big( \mathcal{H}_{16, \, \lambda} \big) =
k (1 - gG)$
\end{center}

$(3)$ Up to an isomorphism of Hopf algebras, there are only three
Hopf algebras that factorize through $\mathbb{H}_4$ and $H_4$,
namely
\begin{equation}\label{cele 3 algebre}
\mathbb{H}_4 \ot H_4, \qquad {\mathcal H}_{16, \, 0} \qquad {\rm
and} \qquad {\mathcal H}_{16, \, 1} \cong D(H_4)
\end{equation}
where $D(H_4)$ is the Drinfel'd double of $H_4$.
\end{theorem}

\begin{proof}
$(1)$ The Hopf algebra ${\mathcal H}_{16, \, \lambda}$ is the
explicit description of the bicrossed product $\mathbb{H}_4
\bowtie H_4$ associated to the non-trivial matched pair given in
\thref{mpHH00}. In $\mathbb{H}_4 \bowtie H_4$ we make the
canonical identifications: $G = G \bowtie 1$, $X = X \bowtie 1$, $
g = 1 \bowtie g$, $x = 1 \bowtie x$. The defining relations of
${\mathcal H}_{16, \, \lambda}$ follow easily. For instance:
\begin{eqnarray*}
x \, X &=& (1 \bowtie x) (X \bowtie 1) = (\lambda - \lambda G)
\bowtie 1 - X \bowtie x + G \bowtie (\lambda - \lambda \, g)\\
&=& \lambda 1 \bowtie 1 - X \bowtie x - \lambda \, G \bowtie g =
\lambda \, 1 - Xx - \lambda \, Gg
\end{eqnarray*}

$(2)$ $\mathcal {H}_{16, \, \lambda}$ is pointed because, as a
coalgebra, is the tensor product of two pointed coalgebras
\cite[Lemma 5.1.10]{Montgomery}. The coradical of
$\mathcal{H}_{16, \, \lambda}$ is $k[\mbox{G}(\mathbb{H}_{4})] \ot
k[\mbox{G}(H_{4})]$, hence $\mbox{G} (\mathcal{H}_{16, \,
\lambda}) = \{ 1, g, G,$ $gG \} \simeq \mathbb{Z}_{2} \times
\mathbb{Z}_{2}$. Since $H_{4}$ is non-semisimple, so is ${\mathcal
H}_{16, \, \lambda}$ \cite[Corollary 3.2.3]{Montgomery}.
${\mathcal H}_{16, \, \lambda}$ is unimodular since $(X + GX)(x -
gx)$ is simultaneously a non-zero left and right integral as it
can easily be verified. The last part follows by a routine check.

$(3)$ $\mathcal{H}_{16, \, \lambda} \cong \mathcal{H}_{16, \, 1}$,
for $\lambda \in k^{\times}$, since the defining relations for
${\mathcal H}_{16, \, 1}$ can be obtained from that of ${\mathcal
H}_{16, \, \lambda}$ by replacing $X$ with $\lambda^{-1} X$.

We next prove that $\mathbb{H}_4 \ot H_4$, $\mathcal{H}_{16, \,
0}$ and $\mathcal{H}_{16, \, 1}$ are non-isomorphic Hopf algebras.
Observe first that $\mathbb{H}_{4} \ot H_{4}$ is generated as an
algebra by the two copies of Sweedler's Hopf algebra,
$\mathbb{H}_{4}$ and $H_{4}$, such that the generators of
$\mathbb{H}_{4}$ commute with the generators of $H_{4}$. Moreover,
the sets of skew-primitive elements of $\mathbb{H}_{4} \ot H_{4}$
have the same description as in (2). In order to distinguish
between the generators of the Hopf algebras in question, we attach
a prime sign, $'$, to the elements of $\mathcal{H}_{16, \, 0}$ and
two such signs to the elements of $\mathcal{H}_{16, \, 1}$.

Assume $\varphi : \mathbb{H}_4 \ot H_4 \to \mathcal{H}_{16, \, 0}$
is a Hopf algebra isomorphism. Then $\varphi (g)$ is a group-like
element of $\mathcal{H}_{16, \, 0}$. Since the vector space of
$(1,g)$-primitive elements must have the same dimension as the
vector space of $\big( 1, \varphi(g) \big)$-elements, we have
$\varphi (g) \in \{g', G'\}$. If $\varphi (g) = g'$ then $\varphi
(G) = G'$ and $\varphi (X) \in \textnormal{P}_{1, G'} \big(
\mathcal{H}_{16, \, 0} \big)$. Let $a$, $b \in k$ such that
$\varphi (X) = a (1 - G') + b X'$. Taking into account that
$\varphi (G) \varphi (X) \varphi (G) = - \varphi (X)$, we obtain
that $a = 0$, hence $\varphi (X) = bX'$. Since $g$ and $X$
commute, we have $g' X' = \varphi (g) \varphi (X) = \varphi (X)
\varphi (g) = b X' g' = - b g'X'$, hence $b = 0$. Thus $\varphi
(X) = 0$, a contradiction with the fact that $\varphi$ has a
trivial kernel. A similar contradiction is obtained if $\varphi
(g) = G'$, so we conclude that $\mathbb{H}_4 \ot H_4$ and
$\mathcal{H}_{16, \, 0}$ are not isomorphic. Since the same
argument works when we consider $\mathbb{H}_4 \ot H_4$ and
$\mathcal{H}_{16, \, 1}$, we deduce that these Hopf algebras are
not isomorphic also.

Suppose now that $\varphi : \mathcal{H}_{16, \, 0} \to
\mathcal{H}_{16, \, 1}$ is a Hopf algebra isomorphism. Then, as
above, $\varphi (g') \in \{g'', G''\}$. If $\varphi (g') = g''$
then $\varphi (G') = G''$, $\varphi (x') \in \textnormal{P}_{1,
g''} \big( \mathcal{H}_{16, \, 1} \big)$ and $\varphi (X') \in
\textnormal{P}_{1, G''} \big( \mathcal{H}_{16, \, 1} \big)$. Let
$a$, $b \in k$ such that $\varphi (x') = a (1 - g'') + b x''$.
Since $\varphi (g) \varphi (x) \varphi (g) = - \varphi (x)$ it
follows that $\varphi (x') = bx''$. A similar argument shows that
$\varphi (X') = dX''$, for some $d \in k$. Using the fact that
$x'X' + X'x' = 0$, we have
$$
0 = \varphi (x'X' + X'x') = bd (x''X'' + X''x'') = bd (1 - g''G'')
$$
Therefore $b = 0$ or $d = 0$, with either case leading to a
contradiction. If $\varphi (g') = G''$ then we arrive at a similar
contradiction so we conclude that $\mathcal{H}_{16, \, 0} \ncong
\mathcal{H}_{16, \, 1}$ .

Finally, we show that $\mathcal{H}_{16, \, 1} \cong D (H_4)$.
First, recall that $D(H_{4})$ factorizes through $\left( H_4^*
\right)^{\rm cop} $ and $H_{4}$. Also, if $\{ 1^*, g^*, x^*,
(gx)^* \}$ denotes the dual basis of $\{1, g, x, gx\}$ then
$H_4^*$ is generated as an algebra by the group-like element $G =
1^* - g^*$ and by the $(G, 1)$-primitive element $X = x^*
+(gx)^*$, with the following relations $G^2 = 1$, $X^2 = 0$, and
$GX = -XG$. Therefore, $\left( H_4^* \right)^{\rm cop} \simeq
\mathbb{H}_4$, so $D (H_4)$ factorizes through $\mathbb{H}_4$ and
$H_4$. In order to see which of the three Hopf algebras from
(\ref{cele 3 algebre}) $D(H_4)$ is, recall the Drinfel'd double as
a matched pair. If $H$ is a finite dimensional Hopf algebra, then
$\left( \left( H^* \right)^{\rm cop}, H, \rhd, \lhd \right)$ is a
matched pair, where $h \rhd h^* = h^* \left(
S^{-1}(h_{(2)})?h_{(1)} \right)$ and $h \lhd h^* = h^* \left(
S^{-1}(h_{(3)})h_{(1)} \right) h_{(2)}$, for all $h \in H$, $h^*
\in \left( H^* \right)^{\rm cop}$, and $D(H) \simeq \left( H^*
\right)^{\rm cop} \bowtie H$. In our case, we have
\begin{align*}
x \lhd X & = X \left( S^{-1}(1)x \right) + X \left( S^{-1}(1)g
\right) x + X \left( S^{-1}(x)g \right) g\\
& = X(x) + X(g)x + X(-x)g\\
& =1 - g
\end{align*}
which shows that $D(H_4) \simeq \mathbb{H}_4 \bowtie_1 H_4$.
\end{proof}

\begin{remark} $\mathcal{H}_{16,0}$ is not the dual of $D(H_4)$ for
otherwise $D(H_4)^*$ would be unimodular and so would $H_{4}$
(\cite{Radford_MQHA}, Corollary 4).
\end{remark}

\begin{remark}
The classification of pointed Hopf algebras of dimension 16 over
an algebraically closed field of characteristic zero was done in
\cite{Caenepeel Dascalescu}. With the notations of \cite[Theorem
5.2]{Caenepeel Dascalescu}, we have $\mathbb{H}_{4} \ot H_{4}
\simeq H_{(3)}$, $\mathcal{H}_{16,0} \simeq H_{(4)}$ and
$\mathcal{H}_{16,1} \simeq H_{(5)}$.
\end{remark}

As a consequence of \thref{clH_4H_4} we are able to describe the
group of Hopf algebra automorphisms of the three Hopf algebras
from (\ref{cele 3 algebre}). We begin with the Drinfel'd double,
$D(H_4)$.

\begin{corollary}\prlabel{autDH}
Let $k$ be a field of characteristic $\neq 2$. Then there exists
an isomorphism of groups
$$
\Aut _{\rm Hopf}(D(H_4)) \cong k^* \rtimes_f \mathbb{Z}_2
$$
where $k^* \rtimes_f \, \mathbb{Z}_2$ is the semidirect product
associated to the action as automorphisms $f : \mathbb{Z}_2 \to
\Aut(k^*)$, $f (1 + 2\mathbb{Z}) (\alpha) = \alpha^{-1}$, for all
$\alpha \in k^*$.
\end{corollary}

\begin{proof}
We use \thref{clH_4H_4} and the description of
$\mathcal{H}_{16,1}$ given in (1). Let $\phi$ be a Hopf algebra
automorphisms of $\mathcal{H}_{16,1}$. Then $\{\phi (g), \phi(G)\}
= \{g, G\}$. If $(\phi (g), \phi(G)) = (g, G)$ then $\phi (x) \in
\textnormal{P}_{1, g} (\mathcal{H}_{16,1})$ and $\phi (X) \in
\textnormal{P}_{1, G} (\mathcal{H}_{16,1})$. Let $a$, $b$, $c$, $d
\in k$ such that $\phi (x) = a (1 - g)  + bx$ and $\phi (X) = c (1
- G) + d X$. Taking into account that $\phi(gxg) = - \phi(x)$ and
$\phi(GXG) = - \phi(X)$ we find that $a = c = 0$. Thus, $\phi (x)
= bx$ and $\phi (X) = dX$. Considering now the relation $xX + Xx =
1 - gG$ and applying $\phi$ to both terms of the equation we find
that $bd = 1$. Thus, $d = b^{-1}$. Since there are no further
restrictions on $b$ imposed by the fact that $\phi$ is a Hopf
algebra isomorphism, we have obtained, for each $b \in
k^{\times}$, an element $\varphi_{b} \in \Aut \big(
\mathcal{H}_{16,1} \big)$, given by
$$
\varphi_{b} (g) = g, \qquad \varphi_{b} (G) = G, \qquad
\varphi_{b} (x) = bx, \qquad \varphi_{b} (X) = b^{-1}X
$$
If $(\phi (g), \phi(G)) = (G, g)$ then, by a similar reasoning as
the one above, we have $\phi (x) = dX$ and $\phi (X) = d^{-1} x$,
for some $d \in k^{\times}$, and each such $\phi$ is a Hopf
algebra isomorphism of $\mathcal{H}_{16,1}$ that we denote by
$\psi_{d}$.

Summarizing, we have obtained that the set of Hopf algebra
automorphisms of $\mathcal{H}_{16,1}$ is $\Aut_{\rm Hopf} \big(
\mathcal{H}_{16,1} \big) = \{ \varphi_{b} \mid b \in k^{\times} \}
\cup \{ \psi_{d} \mid d \in k^{\times} \}$, a disjoint union of
two sets indexed by $k^{\times}$. The elements of $\Aut_{\rm Hopf}
\big( \mathcal{H}_{16,1} \big)$ multiply according to the
following rules
$$
\varphi_{b} \varphi_{d} = \varphi_{b d}, \qquad \psi_{b} \psi_{d}
= \varphi_{b^{-1} d}, \qquad \psi_{b} \varphi_{d} = \psi_{b d},
\qquad \varphi_{b} \psi_{d} = \psi_{b^{-1} d}
$$
for all $b$, $d \in k^{\times}$. Let $k^{\times} \rtimes_f
\mathbb{Z}_2$ be the semidirect product associated to $f :
\mathbb{Z}_2 \to \Aut(k^{\times})$, $f (1 + 2\mathbb{Z}) (b) =
b^{-1}$, for all $b \in k^{\times}$. Taking into account that
multiplication in $k^{\times} \rtimes_f \mathbb{Z}_2$ is given by
$(b, \hat{m}) \cdot (d, \hat{n}) = (b f(\hat{m})(d), \hat{m} +
\hat{n})$, for all $b$, $d \in k^{\times}$ and $\hat{m}$, $\hat{n}
\in \mathbb{Z}_2$, it is easy to see that
$$
\Gamma : k^{\times} \rtimes_f \mathbb{Z}_2 \to \Aut_{\rm Hopf}
\big( \mathcal{H}_{16,1} \big), \qquad \Gamma(b, \hat{0}) =
\varphi_{b}, \qquad \Gamma(b, \hat{1}) = \psi_{b^{-1}}
$$
for all $b \in k^{\times}$, is an isomorphisms of groups.
\end{proof}

\begin{proposition}\prlabel{autH_0}
Let $k$ be a field of characteristic $\neq 2$. Then there exist
isomorphisms of groups
$$
\Aut _{\rm Hopf} \big( \mathcal{H}_{16,0} \big) \cong (k^{\times}
\times k^{\times}) \rtimes_g \mathbb{Z}_2 \cong \Aut _{\rm Hopf}
\big( \mathbb{H}_4 \otimes H_4 \big)
$$
where $(k^{\times} \times k^{\times}) \rtimes_g \, \mathbb{Z}_2$
is the semidirect product associated to the action as
automorphisms $g : \mathbb{Z}_2 \to \Aut \big(k^{\times} \times
k^{\times} \big)$, $g (1 + 2\mathbb{Z}) (b, d) = (d, b)$, for all
$(b, d) \in k^{\times} \times k^{\times}$.
\end{proposition}

\begin{proof}
The method of proof in both cases is the same as the one used in
proving \prref{autDH}. Moreover, both isomorphisms are obtained in
exactly the same way, therefore we will limit ourselves in
pointing out the one for $\mathcal{H}_{16,0}$.

The set of Hopf algebra isomorphisms of $\mathcal{H}_{16,0}$ is
$\Aut_{\rm Hopf} \big( \mathcal{H}_{16,0} \big) = \{ \varphi_{b,
d} \,|\, (b, d) \in k^{\times} \times k^{\times} \} \cup \{
\psi_{b, d} \,|\, (b, d) \in k^{\times} \times k^{\times} \}$,
where $\varphi_{b, d}$ and $\psi_{b, d}$ are defined by
$$
\varphi_{b, d} (g) = g, \qquad \varphi_{b, d} (G) = G, \qquad
\varphi_{b, d} (x) = bx, \qquad \varphi_{b, d} (X) = dX
$$
$$
\psi_{b, d} (g) = G, \qquad \psi_{b, d} (G) = g, \qquad \psi_{b,
d} (x) = bX, \qquad \psi_{b, d} (X) = dx
$$
The elements of $\Aut_{\rm Hopf} \big( \mathcal{H}_{16,0} \big)$
multiply according to the following rules
$$
\psi_{b, d} \psi_{c, e} = \varphi_{d c, b e}, \quad \varphi_{b, d}
\varphi_{c, e} = \varphi_{b c, d e}, \quad \psi_{b, d} \varphi_{c,
e} = \psi_{b c, d e}, \quad \varphi_{b, d} \psi_{c, e} = \psi_{d
c, b e}
$$
for all $b$, $d$, $c$, $e \in k^{\times}$. Considering the
semi-direct product $(k^{\times} \times k^{\times}) \rtimes_g
\mathbb{Z}_2$ associated to $g : \mathbb{Z}_2 \to \Aut \big(
k^{\times} \times k^{\times} \big)$, $g (1 + 2\mathbb{Z}) (b, d) =
(d, b)$, for all $(b, d) \in k^{\times} \times k^{\times}$, we
obtain that
$$
\Gamma : (k^{\times} \times k^{\times}) \rtimes_g \mathbb{Z}_2 \to
\Aut_{\rm Hopf} \big( \mathcal{H}_{16,0} \big), \quad \Gamma \big(
(d, b), \hat{1} \big) = \psi_{b, d}, \quad \Gamma \big( (b, d),
\hat{0} \big) = \varphi_{b, d}
$$
for all $\alpha$, $\beta \in k^{\times}$, is an isomorphism of
groups.
\end{proof}

\section{Acknowledgement}

I would like to thank Professor Gigel Militaru for suggesting the
problem and for his guidance when writing this paper and also the
referee for helpful comments, some of which led to a
simplification of the exposition.


\end{document}